\documentclass{amsart}
\usepackage{color,graphicx,enumerate,amssymb}
\usepackage{hyperref}
\usepackage{epsfig,wrapfig}
\usepackage{pxfonts}
\usepackage{graphicx}
\usepackage{eucal}
\usepackage[utf8]{inputenc}

\usepackage{amsmath}
\usepackage{graphicx}

\newcommand{\bR}{{\mathbb R}}

\newcommand{\cK}{{\mathcal K}}

\newcommand{\cR}{{\mathcal R}}

\newcommand{\cH}{{\mathcal H}}
\newcommand{\cL}{{\mathcal L}}

\newcommand{\cO}{{\mathcal O}}

\newcommand{\uu}{\mathbf{ u}}
\newcommand{\vv}{\mathbf{ v}}
\newcommand{\ww}{\mathbf{ w}}

\newcommand{\ee}{\mathbf{ e}}
\newcommand{\Aa}{\mathbf{ a}}

\newcommand{\g}{\mathbf{ g}}
\newcommand{\e}{\varepsilon}
\newcommand{\kp}{\kappa}

\newcommand{\dist}{{\rm dist}}
\newcommand{\rdiv}{{\rm div\,}}

\newcommand{\ra}{\rightarrow}

\newcommand{\weak}{\rightharpoonup}

\newcommand\norm[1]{\Arrowvert {#1}\Arrowvert}

\usepackage{comment}

\newtheorem{theorem}{Theorem}[section]

\newtheorem{corollary}[theorem]{Corollary}

\newtheorem{definition}[theorem]{Definition}

\newtheorem{lemma}[theorem]{Lemma}

\newtheorem{proposition}[theorem]{Proposition}
\newtheorem{remark}[theorem]{Remark}

\def\Xint#1{\mathchoice
{\XXint\displaystyle\textstyle{#1}}%
{\XXint\textstyle\scriptstyle{#1}}%
{\XXint\scriptstyle\scriptscriptstyle{#1}}%
{\XXint\scriptscriptstyle\scriptscriptstyle{#1}}%
\!\int}
\def\XXint#1#2#3{{\setbox0=\hbox{$#1{#2#3}{\int}$ }
\vcenter{\hbox{$#2#3$ }}\kern-.6\wd0}}

\def\dashint{\Xint-}

\title []{A Minimization Problem with Free Boundary for $p$-Laplacian weakly coupled System}

\author{Morteza Fotouhi  and Henrik Shahgholian }
\address{Department of Mathematical Sciences, Sharif University of Technology, Tehran, Iran}
\email{fotouhi@sharif.edu}
\address{Department of Mathematics, Royal Institute of Technology, 100~44  Stockholm, Sweden}
\email{henriksh@kth.se}

\date{\today}

\begin{document}
    
\begin{abstract} 
In this paper we consider a weakly coupled $p$-Laplacian system of a Bernoulli type free boundary problem, through minimization of a corresponding functional.
We prove various properties of any local  minimizer and the corresponding   free boundary.
\end{abstract}


\thanks{This project was carried out during the program Geometric aspects of nonlinear PDE at Institute Mittag Leffler, Stockholm, Sweden. H. Shahgholian was supported by Swedish Research Council. }

%
%
%

\maketitle


\section{Introduction}\label{introduction}

\subsection{Problem setting}
For  $\Omega\subset\bR^n$ ($n \geq 2$), we 
consider the problem of minimizing the functional
\begin{equation}\label{elliptic-system}
J(\uu)=\int_\Omega\sum_{i=1}^m|\nabla u^i|^p + Q^p\chi_{\{|\uu|>0\}}dx, \qquad 1< p<\infty,
\end{equation}
in the class of vectorial functions 
\[
\cK:=\{\uu=(u^1,\dots,u^m)\in W^{1,p}(\Omega;\bR^m): \uu=\g \text{ on }\partial\Omega\text{ and }u^i\geq0 \text{ for }i=1,\cdots,m\},
\]
with a given  boundary data  $\g\in W^{1,p}(\Omega;\bR^m)$. Here,   $Q$ is a  H\"older  function satisfying 
$$0<Q_{\min}\le Q \leq Q_{\max} < \infty , $$ 
for some constants $Q_{\min}$ and $Q_{\max}$.

We are interested in regularity properties of minimizers $\uu$, as well as the free boundary $\Gamma=\partial\{|\uu|>0\}\cap\Omega$.
Any local  minimizer   satisfies the $p$-Laplace equation
\[
\rdiv(|\nabla u^i|^{p-2}\nabla u^i)=0, \qquad \hbox{ in }\ \{u^i>0\},
\]
also denoted by $\Delta_p u^i = 0$.
In fact, $\Delta_p u^i$ is a nonnegative Radon measure with support on free boundary, $\Gamma$. The problem is to find a reasonable representation of this  measure and put it into some pde-context for further analysis.

This  problem is referred to  as {\it Bernoulli-type} free boundary problem, and is well studied in the literature, for 
the scalar case and for $p=2$, starting with  seminal work  of H.W. Alt and L.A. Caffarelli  \cite{alt1981existence}, and  also for any $1<p<\infty$ in \cite{danielli2005minimum}. 
There are  very few results for Bernoulli-type problems that involve systems (see \cite{caffarelli2018minimization, de2020improvement, mazzoleni2020regularity}). 
In \cite{caffarelli2018minimization}, the authors study the minimum problem \eqref{elliptic-system} for $p=2$ and show the smoothness of  the regular part of free boundary  as  well as some partial result for Hausdorff dimension of singular part. Indeed, they 
apply a reduction method to reduce the problem to its scalar counterpart and the same result for the scalar case can be extended to the vectorial problem. 
Also, a vectorial Bernoulli problem with no sign assumption on the components is studied in \cite{mazzoleni2020regularity}. 
In \cite{de2020improvement}, the same result has been obtained by the viscosity approach and improvement of flatness.

In this paper, we deal with a (weakly coupled) cooperative system for $p$-laplacian version of Bernoulli-type problem, following similar procedure as that in \cite{caffarelli2018minimization}.

\begin{remark}
    It should be remarked that our approaches  in this paper, with some extra efforts, can  be adapted  to variable exponent case, as well as variable coefficient one. Similar types of results are then expected.
\end{remark}

\subsection{Notation}

For clarity of exposition we shall introduce some notations and definitions, which  are used frequently in this text.

Throughout this paper,  $\bR^n$ will be equipped with the Euclidean inner product $x\cdot y$ and the induced norm $|x|$, $B_r(x_0)$ will denote the open $n$-dimensional ball with  center $x_0$, radius $r$ and its boundary with $\partial B_r(x_0)$. 
In addition, $B_r=B_r(0)$ and $\partial B_r=\partial B_r(0)$.
For the target space, $\bR^m$, we use several norms 
\[\begin{split}
&|\uu|_p=((u^1)^p+\dots+(u^m)^p)^{1/p},\\
&|\uu(x)|_\infty=\max_{1\le j\le m} |u^j(x)|.
\end{split}
\]
For convenience, we denote the Euclidean norm without the index, $|\uu|=|\uu|_2$.
We also use the Euclidean norm in the definition of $L^\infty$-norm, that is
\[
\norm{\uu}_{L^\infty}:=\sup_{x}|\uu(x)|.
\]

\subsection{Plan of the paper}
The paper is organized as follows: 
In Section \ref{sec:existence}, we study the existence of minimizer (Theorem \ref{Existence-minimizer}) 
and show that  minimizers are $p$-subharmonic (Lemma \ref{rem:coopertive-property}).
Section \ref{sec:regular-minimizer} is devoted to the regularity property of solutions, including the  H\"older regularity   (Lemma \ref{lem:Holder-regularity}) and  the Lipschitz regularity (Theorem \ref{thm:Liptschitz-regularity}).
Section \ref{sec:nondeg} consists of the proof of nondegeneracy property (Lemma \ref{lem:nondegeneracy}). 
Also, in Theorem \ref{thm:porosity} an estimate for the density of the  free boundary is obtained which is enough to prove that the free boundary has zero Lebesgue measure.
The vector-valued measure $\Delta_p\uu$ (Theorem \ref{thm:measure-laplacian}) and $(n-1)$-Hausdorff dimension of free boundary (Theorem \ref{thm:Hausdorf-measure-fb}) are discussed in Section \ref{sec:measure}.
The main result in Section \ref{sec:analysis} is the flatness of regular part of the free boundary (Theorem \ref{thm:identification-q}). 
We prove a partial result for the regularity of free boundary in Section \ref{sec:FB} (Theorem \ref{thm:reg-free-boundary}), along with   $C^{1,\alpha}$-regularity of the free boundary when $p$ is sufficiently close to 2 (Theorem \ref{thm:reg-free-boundary-2}). 
In Appendix \ref{sec:NTA} we deal with NTA properties of the free boundary.
Also, in Appendix B we present an auxiliary lemma to study the asymptotic behaviour of $p$-harmonic functions.

\section{Existence of a  minimizer}\label{sec:existence}
\begin{theorem}\label{Existence-minimizer}
If $J(\g)<\infty$, then there exists an absolute minimizer of $J$ over class $\cK$. 
\end{theorem}

\begin{proof}
Obviously  the functional is nonnegative, and hence it takes  an infimum value.
Let $\uu_k$ be a minimizing sequence
\[
\inf_{\vv\in\cK}J(\vv)=\lim_{k\ra\infty}J(\uu_k).
\]
Then $\uu_k-\g$ is bounded in $W^{1,p}_0(\Omega;\bR^m)$ and up to a subsequence we can assume that 
\[\begin{split}
 \uu_k\weak \uu, &\quad\text{ weakly in }W^{1,p}(\Omega;\bR^m),\\
 \uu_k\ra\uu, &\quad\text{ a.e. in }\Omega,
\end{split}\]
for some $\uu\in\cK$.
The latter convergence implies 
\[
\int_\Omega\chi_{\{|\uu_k|>0\}} dx \ra \int_\Omega\chi_{\{|\uu|>0\}} dx,
\]
and the weakly lower semicontinuity of the norm implies  that
\[
\int_\Omega\sum_{i=1}^m|\nabla u^i|^p dx \leq \liminf_{k\ra \infty}\int_\Omega\sum_{i=1}^m|\nabla u_k^i|^p dx.
\]
These together show that $J(\uu)\le \lim_{k\ra\infty} J(\uu_k)$ and hence  $\uu\in\cK$ is an absolute minimizer.
\end{proof}

\begin{remark}
We say $\uu\in\cK$ is a {\it local} minimizer of $J$, if $J(\uu)\le J(\vv)$ for any $\vv\in \cK$ with
\[
\norm{\nabla \uu-\nabla\vv}_{L^p(\Omega)}+\norm{\chi_{\{|\uu|>0\}}-\chi_{\{|\vv|>0\}}}_{L^1(\Omega)}\le \e,
\]
for some $\e>0$. 
Although, all results in this paper are proved for local minimizers, for the  sake of convenience we argue with absolute minimizers.
\end{remark}

\begin{lemma}\label{rem:coopertive-property}
If $\uu$ is a (local) minimizer, then $u^i$ is p-subharmonic for all $i=1,\dots,m$, i.e., 
\[
\int_\Omega |\nabla u^i|^{p-2}\nabla u^i\cdot \nabla\varphi\,dx \le 0,\quad \text{ for all }\varphi\in C_0^\infty(\Omega),\,\varphi\ge0.
\]
Moreover, in each component of $\{|\uu|>0\}$ for $i=1,\dots,m$, either $u^i$ is identically vanishing or it is positive. Hence,
\begin{equation}\label{p-laplacian-components}
\Delta_p u^i=0, \quad\text{ in }\{|\uu|>0\}, \, \text{ for all }i=1,\dots,m,
\end{equation}
and consequently (by the maximum principle)
 $$\norm{\uu}_{L^\infty(\Omega)}\leq \norm{\g}_{L^\infty(\Omega)}.$$
\end{lemma}

\begin{proof}
Let $$\varphi\in C_0^\infty(\Omega),\,\varphi\ge 0,$$
$t >0 $,  and 
define $v^j=u^j$ for $j\ne i$ and $v^i=(u^i-t\varphi)_+$. 
Then $\vv\in\cK$ and 
 we can choose $\vv$ as a competitor, so 
\[\begin{split}
0\le \frac1t\left(J(\vv)-J(\uu)\right) = &\frac1t \int_\Omega |\nabla v^i|^p-|\nabla u^i|^p + Q^p\left(\chi_{\{|\vv|>0\}} - \chi_{\{|\uu|>0\}} \right)dx \\
\le &\frac1t \int_\Omega |\nabla (u^i-t\varphi)|^p-|\nabla u^i|^p dx \\
\ra &  \int_\Omega -p |\nabla u^i|^{p-2}\nabla u^i\cdot \nabla\varphi\,dx.
\end{split}\]

The second statement in the  lemma relies on the strong minimum principle for $p$-harmonic functions.
In fact, let $x\in \{|\uu|>0\}$ but $u^i(x)=0$ for some $i$. Choose an index $j$ such that $u^j>0$ inside $B_r(x)$ where $r$ is small enough.
Now denote by $\tilde u^i$ the $p$-harmonic extension of $u^i$ inside $B_r(x)$ and consider the competitor $\tilde \uu $ by defining $\tilde u^k=u^k$ for $k\ne i$.
Since $\{|\uu|>0\}=\{|\tilde\uu|>0\}$, we get  $J(\tilde\uu )< J(\uu)$ unless  
$\tilde u^i=u^i$ and $u^i$ is $p$-harmonic in $B_r(x)$. But this violates the strong minimum principle if $u^i (x) =0$.
\end{proof}


\section{Regularity of local minimizers}\label{sec:regular-minimizer}

\begin{lemma}\label{harmonic-extension}
Let $\uu$ be a (local) minimizer of $J$, and  $v^i$ the harmonic replacement (majorant) for $u^i$ in  $B\subset\Omega$ (for $B$ a small ball).
Then there is a universal constant $C=C(n,p)$ such that 
\begin{align*}
&\int_B|\nabla (u^i- v^i)|^p\,dx \le CQ_{\max}^p|\{|\uu|=0\}\cap B|,\qquad\text{ when }2 \le p,\\
&\int_B|\nabla (u^i- v^i)|^p\,dx \le C(Q_{\max})^{p^2/2}|\{|\uu|=0\}\cap B|^{p/2}\left(\int_B|\nabla u^i|^p\,dx\right)^{1-p/2}, \text{ when }1<p\le 2.
\end{align*}
\end{lemma}
\begin{proof}
Let $v^j=u^j$ for all $j\ne i$, and extend $v^i$ by $u^i$ in $\Omega\setminus B$. 
If $B$ is small enough (when $\uu$ is absolute minimizer, we do not need this assumption), then we have $J(\uu)\le J(\vv)$ and consequently 
\[
\int_B|\nabla u^i|^p -|\nabla v^i|^p  \,dx\le Q_{\max}^p\left|\{|\uu|=0\}\cap B\right|.
\]
Set now $w_s(x)=s u^i(x)+(1-s) v^i(x)$ for $0\le s\le 1$. Then
\[\begin{split}
\int_B|\nabla u^i|^p -|\nabla v^i|^p  \,dx =& \int_B|\nabla w_1|^p -|\nabla w_0|^p  \,dx \\
=& p \int_0^1ds \int_B|\nabla w_s|^{p-2}\nabla w_s\cdot\nabla(u^i-v^i)\, dx \\
=& p \int_0^1ds \int_B\left(|\nabla w_s|^{p-2}\nabla w_s-|\nabla v^i|^{p-2}\nabla v^i\right)\cdot\nabla(u^i-v^i)\, dx \\
=& p \int_0^1\frac{ds}s \int_B\left(|\nabla w_s|^{p-2}\nabla w_s-|\nabla v^i|^{p-2}\nabla v^i\right)\cdot\nabla(w_s-v^i)\, dx,
\end{split}\]
where for  the third equality  we have used  $\Delta_p v^i=0$.
Next  using 
\begin{equation*}
\left(|b|^{p-2}b-|a|^{p-2}a\right)\cdot(b-a)\ge \gamma\left\{
\begin{array}{ll}
|b-a|^2(|b|+|a|)^{p-2},&1<p\le 2,\\[8pt]
|b-a|^p,& 2\le p,
\end{array}\right.
\end{equation*}
we obtain for $p\ge2$
\[
\int_B|\nabla u^i|^p -|\nabla v^i|^p  \,dx \geq \gamma p\int_0^1\frac{ds}s \int_B|\nabla(w_s-v^i)|^p\, dx 
=\gamma p\int_0^1 s^{p-1}ds\int_B|\nabla(u^i-v^i)|^p\, dx,
\]
which shows  the desired estimate for $p\geq 2$.

In case $1< p\le 2$, we have 
\[\begin{split}
\int_B|\nabla u^i|^p -|\nabla v^i|^p  \,dx &\geq \gamma p\int_0^1\frac{ds}s \int_B|\nabla(w_s-v^i)|^2\left(|\nabla w_s|+|\nabla v^i| \right)^{p-2}\, dx \\
&\ge\gamma p\int_0^1 sds\int_B|\nabla(u^i-v^i)|^2\left(s|\nabla u^i|+(2-s)|\nabla v^i| \right)^{p-2}\, dx,\\
&\ge C\int_B|\nabla(u^i-v^i)|^2\left(|\nabla u^i|+|\nabla v^i| \right)^{p-2}\, dx.
\end{split}\]
On the other hand, using the H\"older inequality, we have
\[\begin{split}
\int_B|\nabla (u^i- v^i)|^p\,dx \le \left(\int_B|\nabla(u^i-v^i)|^2\left(|\nabla u^i|+|\nabla v^i| \right)^{p-2}\, dx\right)^{p/2}
\left(\int_B\left(|\nabla u^i|+|\nabla v^i| \right)^{p}\right)^{1-p/2}.
\end{split}\]
We conclude the proof by  applying  $\int_B|\nabla v^i|^p \le \int_B|\nabla u^i|^p $ (since $v^i$ is $p$-harmonic).
\end{proof}

\begin{lemma} (H\"older regularity)\label{lem:Holder-regularity}
Let $\uu$ be a (local) minimizer of $J$ in $B_1$.
Then  for some $\alpha=\alpha(n,p)$
\[
\norm{\uu}_{C^\alpha(B_{3/4})}\le C(n,p,\norm{\uu}_{L^\infty(B_1)}).
\] 
\end{lemma}

\begin{proof}
Let $M=\norm{\uu}_{L^\infty(B_1)}$ and $B_r=B_r(y)$ for $y\in B_{3/4}$ and $r<1/8$. 
Since $u^i$ is a $p$-subsolution, a Caccioppoli type inequality (see \cite{heinonen2018nonlinear}, Lemma 3.27) implies that
\[
\int_{B_r}|\nabla u^i|^p\,dx\le \frac C{r^p}\int_{B_{2r}}(u^i)^p\,dx\le CM^pr^{n-p}.
\] 
On the other hand, if $v^i$ is the $p$-harmonic replacement of $u^i$ inside $B_r$, we have the gradient estimate (see \cite{lewis1983regularity})
\[
\sup_{B_{r/2}}|\nabla v^i| \le \left(\frac C{r^n}\int_{B_r}|\nabla v^i|^p\,dx\right)^{1/p}\le \frac{CM}r.
\]
Now, let us take some $\rho<r/2$ which will be specified below and apply Lemma \ref{harmonic-extension} in $B_r(y)$
\begin{align*}
\norm{\nabla u^i}_{L^p(B_\rho)}\le & \norm{\nabla (u^i-v^i)}_{L^p(B_\rho)} + \norm{\nabla v^i}_{L^p(B_\rho)} \\
\le & \norm{\nabla (u^i-v^i)}_{L^p(B_r)} + C\rho^{n/p} \norm{\nabla v^i}_{L^\infty(B_{r/2})} \\
\le & C\left\{
\begin{array}{ll}
r^{n/p}+M\rho^{n/p}r^{-1} &\text{ for }2 \le p,\\[7pt]
M^{1-p/2}r^{n/p-1+p/2}  +M\rho^{n/p}r^{-1} &\text{ for }1<p\le 2.
\end{array}\right.
\end{align*}
Thus for $r=\rho^{1-\alpha}$, if we take $\alpha=\alpha(n,p)$ sufficiently small, we obtain
\[
\norm{\nabla u^i}_{L^p(B_\rho)}\le C(M,n,p,Q_{\max})\rho^{n/p-1+\alpha}.
\] 
By virtue of Morrey's theorem (see \cite{maly1997fine}) we conclude the proof of the lemma.
\end{proof}

The next lemma is essential to prove the Lipschitz regularity of the minimizers.

\begin{lemma}\label{lem:uniform-L-infty-norm}
Let $\uu=(u^1,\dots,u^m)$ be a bounded minimizer in $B_1$ and $u^i(0)=0$ for some $1\le i\le m$. Then there exists a constant $C=C(n,p,Q_{\max})>0$ such that 
\[
\norm{u^i}_{L^\infty(B_{1/4})}\le C.
\]
\end{lemma} 

We need to remark that the constant $C$  is independent of the boundary values of $\uu$ on $\partial \Omega$. In other words when going away from a  free boundary, but staying uniformly inside the domain $\Omega$, 
the minimizer cannot grow  too large, regardless of the boundary values.
In other words, for large enough boundary values, the origin cannot be a free boundary point.

\begin{proof}
For  the sake of convenience consider $i=1$.
Towards a contradiction, assume that there is a sequence of bounded solutions $\uu_k$ in $B_1$ such that
\[
\norm{u^1_k}_{L^\infty(B_{1/4})} > k.
\]
Set
\[
d_k(x):=\dist(x,\{u^1_k=0\}) \quad \text{ in }B_1,
\]
and define 
\[
\cO_k:=\left\{x\in B_1: d_k(x) \le (1-|x|)/3 \right\}.
\]
Obviously, $B_{1/4}\subset \cO_k$. We have also 
\[
m_k:=\sup_{\cO_k}(1-|x|)u^1_k(x)\ge \frac34\max_{B_{1/4}}u^1_k >\frac34k,
\]
Since $u^1_k$ is bounded (for fixed $k$), we get $(1-|x|)u^1_k(x)\ra0$ as $|x|\ra1$, and therefore $m_k$ is attained at some point $x_k\in\cO_k$. So,
\[
u^1_k(x_k)=\frac{m_k}{1-|x_k|}\geq m_k>\frac34k.
\]
Now let $y_k\in\partial\{u^1_k>0\}\cap B_1$ be such that $|y_k-x_k|=d_k(x_k)=:\delta_k$, which satisfies 
 $\delta_k\leq (1-|x_k|)/3$ due to $x_k\in\cO_k$.
 This implies that
 \[
 B_{2\delta_k}(y_k)\subset B_1\qquad \text{ and }\qquad B_{\delta_k/2}(y_k)\subset\cO_k.
 \]
 Indeed, if $z\in B_{2\delta_k}(y_k)$,
 \[
 |z|\leq |z-y_k|+|y_k-x_k|+|x_k|\leq 2\delta_k+\delta_k+|x_k| \leq 1,
 \]
 and if $z\in B_{\delta_k/2}(y_k)$,
 \[
 1-|z| \ge 1-|x_k| - |x_k-y_k|-|y_k-z| \ge 1-|x_k| - \delta_k -\delta_k/2 \ge 3\delta_k/2 \ge 3|z-y_k|\ge 3d_k(z).
 \]
 Also, we have $1-|z| \ge (1-|x_k|)/2$ for any $z\in B_{\delta_k/2}(y_k)$. Then
 \[
\frac{ 1-|x_k| }2 \max_{B_{\delta_k/2}(y_k)} u^1_k \le \max_{z\in B_{\delta_k/2}(y_k)} (1-|z|)u^1_k(z) \leq \max_{z\in\cO_k}(1-|z|)u^1_k(z) 
= (1-|x_k|)u^1_k(x_k)
 \]
 or 
 \[
 \max_{B_{\delta_k/2}(y_k)} u^1_k \le 2u^1_k(x_k).
 \]
 Since $B_{\delta_k}(x_k)\subset\{u^1_k>0\}$, then $u^1_k$ is $p$-harmonic inside $B_{\delta_k}(x_k)$, i.e. $\Delta_pu^1_k=0$.
 By the Harnack inequality for $p$-harmonic functions, there is a constant $c=c(n,p)$ such that 
 \[
 \min_{B_{4\delta_k/5}(x_k)}u^1_k \geq c u^1_k(x_k).
 \]
 In particular,
 \[
 \max_{B_{\delta_k/4}(y_k)}u^1_k \geq c u^1_k(x_k).
 \]
 We define the sequence
 \[
 \ww_k(x):=\frac{\uu_k(y_k+(\delta_k/2)x)}{u^1_k(x_k)},
 \]
 whose first component   satisfies 
\begin{equation}\label{lem-dist-to-fb:w-property}
 \max_{B_1}w^1_k\leq 2,\qquad\max_{B_{1/2}}w^1_k\geq c>0, \qquad w^1_k(0)=0.
\end{equation}
 Moreover, $\ww_k$ is a minimizer of 
 \[
 J_k(\ww)=\int_{B_1}\sum_{i=1}^m|\nabla w^i|^p+Q_k^p\chi_{\{|\ww|>0\}} dx,
 \]
 where $Q_k(x)=\frac{\delta_k Q(y_k+(\delta_k/2)x)}{2u_k^1(x_k)}\ra0$.
 Now consider $v_k^1$ to be $p$-harmonic replacement of $u_k^1$ in $B_{3/4}$ and apply Lemma \ref{harmonic-extension}
\begin{equation}\label{lem-dist-to-fb:w&v}
 \int_{B_{3/4}}|\nabla(w_k^1-v_k^1)|^p\,dx \leq C(\max Q_k)^p\ra0,
\end{equation}
 when $2\le p$. Similar statement holds for $1<p\le 2$, we just need to note that $\norm{\nabla w_k^1}_{L^p}$ is uniformly bounded 
 ($w_k^1$ is $p$-subsolution and uniformly bounded in $B_1$). 
 Furthermore, $w_k^1$ and $v_k^1$ are uniformly $C^\alpha$ in $B_{5/8}$ and we can extract a subsequence (still denoted by $w_k^1$ and $v_k^1$) such that $w_k^1\ra w_0$ and $v_k^1\ra v_0$ uniformly in $B_{5/8}$.
 Observe that $\Delta_p v_0=0$ in $B_{5/8}$ and \eqref{lem-dist-to-fb:w&v} implies that $w_0=v_0+c$. Hence, $w_0$ is also $p$-harmonic and by the strong maximum principle, $w_0\equiv0$ in $B_{5/8}$, since $w_0\ge0$ and $w_0(0)=0$.
On the other hand, \eqref{lem-dist-to-fb:w-property} necessitates 
\[
\max_{B_{1/2}} w_0\ge c>0,
\]
which is a contradiction.
\end{proof}

A direct consequence of the above lemma is the following estimate. 

\begin{lemma}\label{corollary:Lipshitz-near-coincidence}
Let $\uu$ be a (local) minimizer in $\Omega$. If $\dist(x_0,\{u^i=0\}) < \frac15\dist(x_0,\partial\Omega)$ then
\[
u^i(x_0) \le 4C\dist(x_0,\{u^i=0\}),
\]
where $C$ is the constant defined in Lemma \ref{lem:uniform-L-infty-norm}.
\end{lemma}

\begin{proof}
Choose $y_0\in \{u^i=0\}$ such that $\dist(x_0,\{u^i=0\})=|x_0-y_0|=d_0$. 
Now apply Lemma \ref{lem:uniform-L-infty-norm} to
\[
\vv(x):=\frac{\uu(y_0+4d_0x)}{4d_0}
\]
to get 
\[
u^i(x_0)\le 4Cd_0.
\]
\end{proof}

With the above two results we will obtain uniform Lipschitz regularity for minimizers.

\begin{theorem}\label{thm:Liptschitz-regularity}
Let $\uu$ be a (local) minimizer in $\Omega$, then $\uu$ is Lipschitz. Moreover, for every $K\Subset\Omega$ such that $K\cap\partial\{u^i>0\}\ne\emptyset$ for some $1\le i\le m$, there is a constant $C=C(n,p,Q_{\max},\dist(K,\partial\Omega),\Omega)>0$ such that
\[
\norm{\nabla u^i}_{L^\infty(K)}\le C.
\]
\end{theorem}

Once again we remark that the constant $C$ does not depend on the boundary values of the minimizer, as long as we stay uniformly inside the domain.

\begin{proof}
{\bf Step 1:}
We show that  $u^i$ is bounded in $K$ with a universal constant $C$ depending on the following ingredients $n,p,Q_{\max},\dist(K,\partial\Omega),\Omega$.
Let $r_0=\frac1{5}\dist(K,\partial\Omega)$ and for any arbitrary point $x\in K$ there is a sequence of points $x=x_0,\dots, x_k\in K$ with (we can assume $K$ is connected, otherwise   replace it with a bigger one which is connected)
\[
x_j\in B_{r_0/2}(x_{j-1}),\quad\text{ for }j=1,\dots,k,
\]
 $B_{r_0}(x_j)\subset\{u^i>0\}$ for $j=0,\dots,k-1$ and $B_{r_0}(x_k)\cap\{u^i=0\}\ne \emptyset$. 
Note that $k$, the number of points, only depends on $\Omega$ and $\dist(K,\partial\Omega)$.
From Lemma \ref{corollary:Lipshitz-near-coincidence}, we get 
\[
u^i(x_k)\leq 4Cr_0.
\]
Since $u^i$ is $p$-harmonic in $B_{r_0}(x_j)$, $j=0,\dots,k-1$, by virtue of Harnack's inequality,  there is a constant $c$ such that
\[
u^i(x_{j+1})\ge c u^i(x_j).
\]
Thus
\[
u^i(x)\le 4c^{-k}Cr_0.
\]
 
{\bf Step 2:}
Here we find a control on $\nabla u^i$ at points close to $\{u^i=0\}$.
If $d=\dist(y,\{u^i=0\}) < \frac1{11}\dist(y,\partial\Omega)$, every points $x_0\in B_{d}(y)$ satisfy condition Lemma \ref{corollary:Lipshitz-near-coincidence}. Then
\[
u^i(x_0)\le 4C\dist(x_0,\{u^i=0\}) \leq 8Cd.
\]
Let us define 
\[
v(x):=\frac{u^i(y+dx)}d
\]
which is a $p$-harmonic in $B_1$ and $\norm{v}_{L^\infty(B_1)}\leq 8C$. 
By $p$-Laplacian estimate for gradient, we obtain 
\[
|\nabla v(0)|\leq \tilde C(n,p,Q_{\max}),
\]
that is $|\nabla u^i(y)|\le C(n,p,Q_{\max})$.

 {\bf Step 3:}
 Let $r_1=\frac1{11}\dist(K,\partial\Omega)$. If $\dist(x,\{u^i=0\})\leq r_1$, by the result of Step 2 we have already $|\nabla u^i(x)| \leq C$.
 If $\dist(x,\{u^i=0\})>r_1$, then $u^i$ is $p$-harmonic inside $B_{r_1}(x)$ and $\norm{u^i}_{L^\infty(B_{r_0})}$ is universally bounded by the result of Step 1. Thus $|\nabla u^i(x)|$ will be universally bounded. 
 \end{proof}

  A straightforward corollary to this theorem, that can be useful later, is the following

 \begin{corollary}\label{cor:nondege-lipschitz} 
 Let $\uu $ be a (local) minimizer for our functional.
 For every $K\Subset\Omega$ there exists constant $C=C(n,p,Q_{\max},\dist(K,\partial\Omega),\Omega)$ such that 

  \[
\frac1r\dashint_{\partial B_r}u^i\,dx > C\ \ \text{ implies }\ \ u^i>0\text{ in }B_r.
 \]
 \end{corollary}
 \begin{proof}
 If $B_r\subset K$ contains a free boundary point, then by Theorem \ref{thm:Liptschitz-regularity}, $u^i\leq Cr$ on $\partial B_r$.
 \end{proof}


 \section{Nondegeneracy}\label{sec:nondeg}

 \begin{lemma}\label{lem:nondegeneracy}
 For any  $0<\kp<1$ there exists a constant $c=c(\kp,n,m,p, Q_{\min})>0$ such that for every minimizer $\uu$ and for any (small) ball $B_r\subset\Omega$
 \[
\norm{\uu}_{L^\infty(B_{r})} < cr\ \ \text{ implies }\ \ \uu= {\bf 0} \text{ in } B_{\kp r}.
 \]
 \end{lemma}
 \begin{proof}
 Without loss of generality, we may assume  $r=1$. Let
 \[
 M=\norm{\uu}_{L^\infty(B_{\sqrt\kp})}.
 \]
 Let $\phi(x)=\phi_\kp(|x|)$ be the solution of
 \[
 \Delta_p\phi=0,\ \text{ in }B_{\sqrt\kp}\setminus B_\kp, \qquad \phi=0\ \text{ on }\partial B_\kp, \qquad \phi=1\ \text{ on }\partial B_{\sqrt\kp}
 \]
 and extend  $\phi=0$ in $B_\kp$. Set $v= M\sqrt\kp \phi$ and $w^i=\min(u^i,v)$ for all $i=1,\dots,m$. 
 Since $v\geq u^i$ on $\partial B_{\sqrt\kp}$, so $w^i=u^i$  on $\partial B_{\sqrt\kp}$. 
 Therefore $J(\uu)\leq J(\ww)$, or equivalently 
 \[
 \int_{B_{\sqrt\kp}} \sum_{i=1}^m|\nabla u^i|^p + Q^p\chi_{\{|\uu|>0\}}\, dx \leq \int_{B_{\sqrt\kp}\setminus B_\kp} \sum_{i=1}^m|\nabla w^i|^p + Q^p\chi_{\{|\ww|>0\}}\, dx.
 \]
 Since $\{|\ww|>0\}\subset \{|\uu|>0\}$, we get
 \[\begin{split}
  \int_{B_{\kp}}\sum_{i=1}^m |\nabla u^i|^p + Q^p\chi_{\{|\uu|>0\}}\, dx 
 & \leq \int_{B_{\sqrt\kp}\setminus B_\kp}\sum_{i=1}^m\left( |\nabla w^i|^p-|\nabla u^i|^p \right)\,dx\\
 &\leq p\int_{B_{\sqrt\kp}\setminus B_\kp}\sum_{i=1}^m|\nabla w^i|^{p-2}\nabla w^i\cdot\nabla(w^i-u^i)\,dx\\
 &=-p\int_{\partial B_\kp}\sum_{i=1}^m|\nabla w^i|^{p-2}(w^i-u^i)(\nabla w^i\cdot\nu)\,d\cH^{n-1}\\
 &=p\int_{\partial B_\kp}\sum_{i=1}^m|\nabla v|^{p-2}u^i(\nabla v\cdot\nu)\,d\cH^{n-1}.
 \end{split}\]
 Since $|\nabla v|\leq C(p,\kp,n)M$ on $\partial B_\kp$, we find out  that
\begin{equation}\label{lem:nondegneracy-rel1}
 \int_{B_{\kp}}\sum_{i=1}^m |\nabla u^i|^p + Q^p\chi_{\{|\uu|>0\}}\, dx \le CM^{p-1}\sum_{i=1}^m\int_{\partial B_\kp}  u^i\,d\cH^{n-1}.
 \end{equation}
 On the other hand,
 \[\begin{split}
 \int_{\partial B_\kp}  u^i\,d\cH^{n-1}&\le C(n,\kp)\int_{B_\kp}u^i+|\nabla u^i|\,dx \\
 &= C(n,\kp)\int_{B_\kp}\left(u^i+|\nabla u^i|\right)\chi_{\{u^i>0\}}  \,dx \\
 &\le C(n,\kp,p,Q_{\min})\int_{B_\kp} MQ^p\chi_{\{u^i>0\}} + |\nabla u^i|^p+Q^p\chi_{\{u^i>0\}} \,dx\\
 &\leq C(n,\kp,p,Q_{\min})(1+M)\int_{B_\kp}|\nabla u^i|^p+Q^p\chi_{\{u^i>0\}} \,dx.
 \end{split}\]
 Comparing with \eqref{lem:nondegneracy-rel1}, we we arrive at 
 \[
 \int_{B_{\kp}}\sum_{i=1}^m |\nabla u^i|^p + Q^p\chi_{\{|\uu|>0\}}\, dx \le CM^{p-1}(1+M) \int_{B_{\kp}}\sum_{i=1}^m |\nabla u^i|^p + Q^p\chi_{\{|\uu|>0\}}\, dx.
 \]
 Therefore, if $M$ is small enough, we obtain that $\uu= {\bf 0}$ in $B_\kp$.
 \end{proof}

 An immediate consequence of the above lemma is the following.  For any $K\Subset\Omega$ there are positive constants $c_0,C_0$ such that if $B_r(x)\subset K\cap\{|\uu|>0\}$ touches $\partial\{|\uu|>0\}$ then
 \begin{equation}\label{corol:nondegeneracy}
      c_0r \le |\uu(x)| \le C_0r.
 \end{equation}

 \medskip
 
 \begin{theorem}\label{thm:porosity}
 For  $K\Subset\Omega$ there exists constant $0<c=c(n,m,p,K,\Omega)<1$ such that for any (local) minimizer $\uu$ and for any (small) ball 
 $B_r(x)\subset K$ with $x\in\partial\{|\uu|>0\}$,
 \begin{equation}\label{thm:porosity-rel1}
 c < \frac{\cL^n(B_r (x) \cap\{|\uu|>0\})}{\cL^n(B_r (x))} < 1-c.
 \end{equation}
 \end{theorem}
 
 \begin{proof}
 By Lemma \ref{lem:nondegeneracy}, there exists $y\in B_{r/2}$ such that $|\uu(y)|\ge cr>0$. 
 Using Lipschitz continuity we get
 \[
 \dashint_{\partial B_{\kp r}(y)}|\uu| \geq \frac{cr}2,
 \]
 provided $\kp$ is small enough. Hence
 \[
\frac1{\kp r} \dashint_{\partial B_{\kp r}(y)}|\uu| \geq \frac{c}{2\kp},
 \]
 and also for at least one component $u^i$
 \[
 \frac1{\kp r} \dashint_{\partial B_{\kp r}(y)}u^i \geq \frac{c}{2\kp m},
 \]
 which  by Corollary  \ref{cor:nondege-lipschitz} implies   $|\uu|>0$ in $B_{\kp r}(y)$. This gives the lower estimate in \eqref{thm:porosity-rel1}.
 
 To prove the estimate from above we assume,   for simplicity, $r=1$ and suppose (towards a contradiction) 
 that  there is a sequence of minimizers $\uu_k$ in $B_1(0)$ such that $0\in \partial\{|\uu_k|>0\}$ and 
 \[
 \cL^n(\{|\uu_k|=0\})=:\e_k\ra0.
 \]
 Let $v_k^i$ be a $p$-harmonic function in $B_{1/2}$ with boundary data $v_k^i=u_k^i$ on $\partial B_{1/2}$.
 From Lemma \ref{harmonic-extension}, we obtain that
 \begin{equation}\label{thm:porosity-rel2}
 \int_{B_{1/2}}|\nabla(v_k^i-u_k^i)|^p\,dx \le C(\e_k)\ra0.
 \end{equation}
 Since $u_k^i$ and $v_k^i$ are both uniformly Lipschitz in $B_{1/4}$, we may assume that $u_k^i\ra u_0^i$ and $v_k^i\ra v_0^i$ uniformly in $B_{1/4}$.
 Observe that $\Delta_p v_0^i=0$ and \eqref{thm:porosity-rel2} implies that $u_0^i=v_0^i+ c$. 
 Thus $\Delta_p u_0^i=0$ in $B_{1/4}$ and from the strong minimum principle (since $u_0^i (0) =0$)  it follows $u_0^i\equiv 0$ in $B_{1/4}$, since $u_0^i\ge0$ and $u_0^i(0)=0$.
 On the other hand form nondegeneracy property, Lemma  \ref{lem:nondegeneracy}, we know 
 \[
\norm{\uu_k}_{L^\infty(B_{1/2})}\ge c>0, 
 \]
 which implies a similar inequality for $\uu_0$, and hence  a contradiction.
 \end{proof}

 \begin{remark}\label{rem:free-bndry-measure-zero}
 Theorem \ref{thm:porosity}, along with the Lebesgue density theorem  implies that the free boundary  has zero Lebesgue measure $$\cL^n(\partial\{|\uu|>0\})=0.$$

 \end{remark}
 
 
 \section{The vector-valued  measure $\Delta_p\uu$}\label{sec:measure}
 
 Let $0\le \zeta \in C^\infty_0(\Omega)$ be a test function, and define the measure $\lambda^i$ by
\[
\int\zeta\,d\lambda^i = -\int |\nabla u^i|^{p-2} \nabla u^i\cdot\nabla\zeta\, dx ,
\]
which in virtue of Lemma \ref{rem:coopertive-property} is a bounded non-negative  measure, i.e. a Radon measure.
Obviously $\lambda^i $ is the formal way of expressing $\Delta_p u^i$ in $\Omega $.

Since each $u^i$ is $p$-subharmonic in $\Omega$ and $u^i\ge0$ we have that $\lambda^i$ is a positive Radon measure. Because $u^i$ is also $p$-harmonic in $\{u^i>0\}$ we have that the support of $\lambda^i$ is in  
$\Omega\cap\partial\{u^i>0\}\subseteq\Omega\cap\partial\{|\uu|>0\}$.\footnote{Observe that $u^i$ may be zero in some component of $\{|\uu|>0\}$.}
Let us define    
\[
\Lambda = \Lambda^{\uu}:=\sum_{i=1}^m\lambda^i.
\]

\begin{theorem}\label{thm:measure-laplacian}
For any $K\Subset\Omega$ there exist constants $c,C>0$ such that for any (local) minimizer $\uu$
\[
cr^{n-1}\le \int_{B_r} d\Lambda \le Cr^{n-1}
\]
for any ball $B_r\subset K$ with $x\in \partial\{|\uu|>0\}$.
\end{theorem}

\begin{proof}
 Let $0\le  \zeta_\epsilon \in C^\infty_0(B_{r+\epsilon})$ be a suitable test function, such that $ \zeta_\epsilon= 1$ on $B_r$ and $|\nabla \zeta_\epsilon|\le 2/\epsilon$.
Then
\[
\int \zeta_\epsilon \,d\lambda^i = -\int |\nabla u^i|^{p-2} \nabla u^i\cdot\nabla  \zeta_\epsilon\,dx
= -\int_{B_{r+ \epsilon} \setminus B_r} |\nabla u^i|^{p-2} \nabla u^i\cdot\nabla  \zeta_\epsilon\,dx \le Cr^{n-1},
\]
where in the last inequality we have used that $u$ is Lipschitz. 
Letting $\epsilon$ tend to zero,  we arrive at 
\[
\int_{B_r}\,d\lambda^i \leq Cr^{n-1}.
\]

To prove the estimate from below, we argue indirectly.
It also suffices  to consider the case $r=1$. Assume there is a sequence of minimizers $\uu_k$ in the unit ball $B_1(0)$, such that $0\in\partial\{|\uu_k|>0\}$ and for the measures $\Lambda_k := \Lambda^{\uu_k}  $ we have 
\[
\e_k:=\Lambda_k(B_1)\ra0.
\]
Since the functions $\uu_k$ are uniformly Lipschitz continuous, we may assume that $\uu_k\ra\uu_0$ in $B_{1/2}$, where $\uu_0$ is Lipschitz continuous as well. 
We may also extract a subsequence (still denote by $\uu_k$) such that $g_k^i:=|\nabla u^i_k|^p\nabla u^i_k\ra g_0^i$ weakly-$*$ in $L^\infty(B_{1/2})$ for all $i=1,\dots,m$. We claim that 
\begin{equation}\label{thm:mesr-lap-claim1}
g_0^i=|\nabla u^i_0|^p\nabla u^i_0,\quad\text{ in }B_{1/2}.
\end{equation}
Suppose this is true, then  for every positive test function $\zeta\in C_0^\infty(B_{1/2})$ one has
\[\begin{split}
-\int_{B_{1/2}}|\nabla u^i_0|^{p-2}\nabla u^i_0\cdot\nabla\zeta\,dx =& -\lim_{k\ra\infty}\int_{B_{1/2}}|\nabla u^i_k|^{p-2}\nabla u^i_k\cdot\nabla\zeta\,dx \\
=&\lim_{k\ra\infty}\int_{B_{1/2}}\zeta\,d\lambda^i_k\leq \norm{\zeta}_{L^\infty(B_{1/2})}\lim_{k\ra\infty}\e_k=0.
\end{split}\]
Thus $\lambda_0^i=0$ and $u^i_0$ is $p$-harmonic for all $i=1,\dots,m$ (note that $u^i_0$ is the limit of a sequence of $p$-subharmonic functions and we  already know that it is $p$-subharmonic).
Since $u_0^i\ge0$ and $u_0^i(0)=0$, by  the minimum  principle, we have 
$u_0^i\equiv0$ in $B_{1/2}$.
 
On the other hand, by nondegeneracy property (Lemma \ref{lem:nondegeneracy}) and that $0\in\partial\{|\uu_k|>0\}$ we have 
\[
\norm{\uu_k}_{L^\infty(B_{1/4})}\geq c>0.
\]
Therefore, a similar inequality holds for $\uu_0$ and we arrive at  a contradiction. 

To close the argument we need to prove \eqref{thm:mesr-lap-claim1}.
In fact, if $\overline{B_\rho}=\overline{B_\rho(y)}\subset\{|\uu_0|>0\}$ then $B_\rho\subset\{|\uu_k|>0\}$ for sufficiently large $k$ and $u_k^i$ are $p$-harmonic in $B_\rho$ for all $i=1,\dots,m$, (see Lemma \ref{rem:coopertive-property}).
Therefore, one can extract a subsequence of $\uu_k$ locally converging to $\uu_0$ in $C^{1,\alpha}(B_\rho)$.
Hence, $g_0^i=|\nabla u^i_0|^p\nabla u^i_0$ in $B_\rho$ for all $i=1,\dots,m$.
Next, if $B_\rho\subset\{|\uu_0|=0\}$ then for any $\kp<1$ the nondegeneracy property entails that $B_{\kp\rho}\subset\{|\uu_k|=0\}$ for sufficiently large $k=k(\kp)$. Thus $g_0^i=0=|\nabla u^i_0|^p\nabla u^i_0$. 
We just need to show $\cL^n(\partial\{|\uu_0|>0\}\cap B_{1/2})=0$. If $x_0\in\partial\{|\uu_0|>0\}\cap B_{1/2}$, then $\uu_0(x_0)=0$. 
Choose $x_k\in \partial\{|\uu_k|>0\}\cap B_{1}$ such that $|x_k-x_0|=\dist(x_0,\partial\{|\uu_k|>0\})$, then relation \eqref{corol:nondegeneracy} yields that $|x_k-x_0|\ra0$. Apply Lemma \ref{lem:nondegeneracy} to obtain
\[
\norm{\uu_k}_{L^\infty(B_{2r}(x_0))} \ge \norm{\uu_k}_{L^\infty(B_r(x_k))} \geq cr,
\]
for any ball $B_r(x_0)\subset B_{1/2}$ and sufficiently large $k$.
Passing to the limit we get  the same inequality for $\uu_0$,
\[
\norm{\uu_0}_{L^\infty(B_{2r}(x_0))}  \geq cr.
\] 
This along with the Lipschitz continuity of $\uu_0$ is enough to prove that $\cL^n(B_r(x_0)\cap\{|\uu_0|>0\})\geq c\cL^n(B_r)$ for some $c>0$. (see the first part of the proof of Theorem \ref{thm:porosity}).
This implies that $\cL^n(\partial\{|\uu_0|>0\}\cap B_{1/2})=0$ (see Remark \ref{rem:free-bndry-measure-zero}).
\end{proof}

The next theorem follows easily from Theorem \ref{thm:measure-laplacian}. The proof is the same as the proof of Theorem 4.5 in \cite{alt1981existence}.

\begin{theorem}\label{thm:Hausdorf-measure-fb}
Let $\uu$ be a (local) minimizer in $\Omega$. Then
\begin{enumerate}[(i)]
\item For every $K\Subset\Omega$ we have $\cH^{n-1}(K\cap\partial\{|\uu|>0\})<\infty$.

\item There exist  nonnegative Borel functions $q^i$ such that 
\[
\Delta_p u^i=q^i\cH^{n-1}\lfloor\,\partial\{|\uu|>0\},
\]
that is for every $\zeta\in C_0^\infty(\Omega)$
\[
-\int_\Omega|\nabla u^i|^{p-2}\nabla u^i\cdot\nabla\zeta\,dx=\int_{\Omega\cap\partial\{|\uu|>0\}}\zeta q^i d\cH^{n-1}.
\]

\item For any $K\Subset\Omega$ there exist constants $c,C>0$ such that 
\[
c\leq \sum_{i=1}^m q^i\leq C,
\]
and for $B_r(x)\subset K$ with $x\in\partial\{|\uu|>0\}$ we have
\[
cr^{n-1}\leq \cH^{n-1}(B_r(x)\cap\partial\{|\uu|>0\})\leq Cr^{n-1}.
\]
\end{enumerate}
\end{theorem}

\begin{remark}
From $(i)$ in Theorem \ref{thm:Hausdorf-measure-fb} it follows that, locally,  the set $A=\Omega\cap \{|\uu|>0\}$  has finite perimeter in $\Omega$ in sense of that  $\mu_{\uu}=-\nabla\chi_A$ is a Borel measure and the total variation $|\mu_\uu|$ is a Radon measure. 
We define the reduced boundary of $A$ by
\[
\partial_{\rm{red}}A=\{x\in \Omega:|\nu_\uu|=1\},
\]
where $\nu_\uu(x)$ is the unique unit vector with
\[
\int_{B_r(x)}\left|\chi_A-\chi_{\{y:(y-x)\cdot\nu_\uu(x)<0\}}\right|=o(r^n),\quad\text{ as }r\ra0,
\]
if such a vector exists, and $\nu_\uu(x)=0$ otherwise.  See \cite{federer2014geometric}, Chapter 4, for more details.
\end{remark}


\section{Local Analysis}\label{sec:analysis}

To proceed, we will need some properties of the so-called blow-up limits.
\begin{lemma}\label{lem:blow-up}
Let $\uu$ be a (local) minimizer in $\Omega$, $K\Subset \Omega$ and $B_{r_k}(x_k)\subset K$ be a sequence of balls with $r_k\ra0$, $x_k\ra x_0\in\Omega$, and $\uu(x_k)=0$.
Consider the blow-up sequence
\begin{equation}\label{blowup1}
\uu_k(x)=\frac1{r_k}\uu(x_k+r_k x).
\end{equation}
For a subsequence, there is a limit $\uu_0$ such that
\begin{align}
\uu_k\ra \uu_0 & \ \text{ in } C^{0,\alpha}_{\rm{loc}}(\bR^n;\bR^m)\text{ for every }0<\alpha<1,\notag\\
\nabla \uu_k\ra \nabla \uu_0 & \ \text{ a.e. in }\bR^n,\label{blow-up-limit-rel2}\\
\partial\{|\uu_k|>0\}\ra\partial\{|\uu_0|>0\}&\ \text{ locally in the Hausdorff distance,}\label{blow-up-limit-rel3}\\
\chi_{\{|\uu_k|>0\}}\ra \chi_{\{|\uu_0|>0\}}&\ \text{ in }L^1_{\rm{loc}}(\bR^n;\bR^m),\label{blow-up-limit-rel4}\\
\text{ if }x_k\in\partial\{|\uu_k|>0\}&\ \text{ then }0\in\partial\{|\uu_0|>0\}.\notag
\end{align}
\end{lemma}
\begin{proof}
For the proof we refer to \cite{alt1981existence} and \cite{alt1984free}. 
\end{proof}

The following lemma shows that the blow-up limit is a minimizer in any ball.

\begin{lemma}\label{lem:blow-up-minimizer}
If $\uu(x_k)=0$ and $x_k\ra x_0$, then any blow-up limit $\uu_0= \lim_k \uu_k$ (see \eqref{blowup1}) with respect to $B_{r_k}(x_k)$ is an absolute minimizer of $J_0$ in any ball $B_R=B_R(0)$, where 
$$J_0(\vv):=\int_{B_R}\sum_{i=1}^m|\nabla v^i|^p+Q(x_0)^p\chi_{\{|\vv|>0\}}\,dx.$$
\end{lemma}

\begin{proof}
Let $\ww\in W^{1,p}(B_R;\bR^m)$ be such that $w^i\ge0$ for $i=1,\dots,m$ and $\ww=\uu_0$ on $\partial B_R$.
To show $J_0(\uu_0)\le J_0(\ww)$, we
choose a cut of function $\eta\in C_0^\infty(B_R)$ with $0\le \eta\le 1$ and $\eta=1$ in $B_r$ for some $0<r<R$, and   define 
\[
\ww_k=\left(\ww+(1-\eta)(\uu_k-\uu_0)\right)_+,
\]
where the positive part is taken separately for each component.
We also  have $\ww_k=\uu_k$ on $\partial B_R$. Since $\uu$ is (local) minimizer, for sufficiently large $k$ such that $B_{Rr_k}(x_k)\Subset\Omega$, we have 
\[
\int_{B_{R}}\sum_{i=1}^m|\nabla u^i_k|^p+Q_k^p\chi_{\{|\uu_k|>0\}}\,dx \leq \int_{B_{R}}\sum_{i=1}^m|\nabla w^i_k|^p+Q_k^p\chi_{\{|\ww_k|>0\}}\,dx,
\]
where $Q_k(x):=Q(x_k+r_kx)$.
Since $|\nabla\uu_k|\leq C$ (due to Lipschitz continuity of $\uu$) and convergences  \eqref{blow-up-limit-rel2} and \eqref{blow-up-limit-rel4}, 
the limit of the  left hand side will be $J_0(\uu_0)$.  Hence  
\begin{align*}
J_0(\uu_0)\le & \liminf_{k\ra\infty} \int_{B_R}\sum_{i=1}^m|\nabla w_k^i|^p + Q_k^p\chi_{\{|\ww_k|>0\}}\,dx\\
\leq & \int_{B_R}\sum_{i=1}^m|\nabla w^i|^p \,dx + \int_{B_r} Q(x_0)^p\chi_{\{|\ww|>0\}}\,dx 
+ \liminf_{k\ra\infty}\int_{B_R\setminus B_r} Q_k^p\chi_{\{|\ww_k|>0\}}\,dx \\
\le &\int_{B_R}\sum_{i=1}^m|\nabla w^i|^p\,dx + \int_{B_r}Q(x_0)^p\chi_{\{|\ww|>0\}}\,dx + Q(x_0)^p|B_R\setminus B_r|.
\end{align*}
Now let $r\ra R$, we get $J_0(\uu_0)\le J_0(\ww)$.
\end{proof}

\begin{lemma}\label{lem:domain-variation-formula}
Suppose $\nabla(Q^p)\in L^1(\Omega)$ and $\uu$ is an absolute minimizer. Then
\[
\int_{|\uu|>0}\rdiv\left[\sum_{i=1}^m \left(|\nabla u^i|^{p}\Psi-p|\nabla u^i|^{p-2}(\nabla u^i\cdot\Psi)\nabla u^i\right)+Q^p\Psi\right]dx=0,
\]
for every $\Psi\in C_c^\infty(\Omega;\bR^n)$. 
\end{lemma}
\begin{proof}
Let us  define 
\[
\Phi_t(x)=x+t\Psi(x) \qquad \text{ and }\qquad \uu_t(x)=\uu(\Phi_t(x)).
\]
One can show that for sufficiently small $|t|$, $\Phi_t:\Omega\ra\Omega$ is a diffeomorphism.
We have $D\Phi_t=I+t D\Psi$ and for $i=1,\dots,m$
\[
\nabla u^i_t=D\Phi_t(x)\nabla u^i(\Phi_t(x)).
\]
It follows that
\[
|\nabla u_t^i(x)|^2=(\nabla u^i(\Phi_t(x)))^TA_t(x)\nabla u^i(\Phi_t(x)),
\]
where 
\[
A_t=(D\Phi_t)^TD\Phi_t=I + t((D\Psi)^T+D\Psi)+t^2(D\Psi)^TD\Psi.
\]
By a change of variables, we have
\[\begin{split}
J(\uu_t)=&\int_\Omega\sum_{i=1}^m|\nabla u^i_t(x)|^p+Q^p\chi_{\{|\uu_t|>0\}}\,dx \\
=&\int_\Omega\sum_{i=1}^m \left((\nabla u^i(\Phi_t(x)))^TA_t(x)\nabla u^i(\Phi_t(x))\right)^{p/2}+ Q^p(x)\chi_{\{|\uu|>0\}}(\Phi_t(x))\,dx \\
=& \int_\Omega\sum_{i=1}^m \left[ \left((\nabla u^i(y))^TA_t(\Phi_t^{-1}(y))\nabla u^i(y)\right)^{p/2}+ Q^p(\Phi_t^{-1}(y))\chi_{\{|\uu|>0\}}(y)\right]\left|\det D_y\Phi_t^{-1}(y)\right|\,dy\\
=&\int_{\{|\uu|>0\}}\sum_{i=1}^m \left[ \left((\nabla u^i(y))^TA_t(\Phi_t^{-1}(y))\nabla u^i(y)\right)^{p/2}+ Q^p(\Phi_t^{-1}(y))\right]\left|\det D_y\Phi_t^{-1}(y)\right|\,dy .
\end{split} \]
We also have 
\[
\frac d{dt}A_t(\Phi_t^{-1}(y))\Big|_{t=0}= D\Psi(y)+D\Psi(y)^T
\]
and 
\[
\frac d{dt}\left|\det D_y\Phi_t^{-1}(y)\right|\Big|_{t=0}= -\rdiv \Psi(y).
\]
Now differentiate $J(\uu_t)$ with respect to $t$ and note that its minimum is attained at $t=0$, then
\[\begin{split}
0=\frac d{dt}J(\uu_t)\Big|_{t=0}=&\int_{\{|\uu|>0\}}\sum_{i=1}^m p|\nabla u^i|^{p-2}(\nabla u^i)^TD\Psi\nabla u^i-\Psi\cdot\nabla(Q^p)dx\\
&-\int_{\{|\uu|>0\}}\left[\sum_{i=1}^m |\nabla u^i|^{p}+Q^p\right]\rdiv\Psi\,dx\\
=&-\int_{\{|\uu|>0\}}\rdiv\left[\sum_{i=1}^m \left(|\nabla u^i|^{p}\Psi-p|\nabla u^i|^{p-2}(\nabla u^i\cdot\Psi)\nabla u^i\right)+Q^p\Psi\right]\\
&\qquad\qquad+\sum_{i=1}^mp(\nabla u^i\cdot\Psi)\Delta_p u^i\, dx .
\end{split}\]
Since each $u^i$ is $p$-harmonic in $\{ u^i > 0\}$, see \eqref{p-laplacian-components}, we arrive at the desired claim, in the lemma.
\end{proof}

\begin{definition}
The upper $\cH^{n-1}$-density at any point $x_0\in\partial\{|\uu|>0\}$ is defined as
\[
\Theta^{*n-1}\left(\cH^{n-1}\big\lfloor\partial\{|\uu|>0\},x_0\right):=\limsup_{r\ra0}\frac{\cH^{n-1}(B_r(x_0)\cap \partial\{|\uu|>0\})}{\omega_{n-1}r^{n-1}},
\]
where $\omega_{n-1}$ denotes  the volume of the unit sphere in $\bR^{n-1}$.  
We already know  (see for example Theorem 2.7 in \cite{evans2018measure}) that for  $\cH^{n-1}$-a.e. point  $x_0\in\partial\{|\uu|>0\}$, their upper $\cH^{n-1}$-density satisfy
\[
\Theta^{*n-1}\left(\cH^{n-1}\big\lfloor\partial\{|\uu|>0\},x_0\right)\leq 1.
\] 
\end{definition}

\begin{theorem}\label{thm:identification-q}
Let $x_0\in \partial_{\rm red}\{|\uu|>0\}$ and suppose that
\[
\Theta^{*n-1}\left(\cH^{n-1}\big\lfloor\partial\{|\uu|>0\},x_0\right)\leq 1.
\]
Then ${\rm Tan}(\partial\{|\uu|>0\},x_0)=\{x:x\cdot\nu(x_0)=0\}$.
If, in addition, $x_0$ is a Lebesgue point for Radon measure $q^i\cH^{n-1}\big\lfloor\partial\{|\uu|>0\}$, that is
\begin{equation}\label{thm:identification-q-assump-lebeg-point}
\int_{B_r(x_0)\cap\partial\{|\uu|>0\}}|q^i-q^i(x_0)|\,d\cH^{n-1}=o(r^{n-1}), \quad\text{ as }r\ra0,
\end{equation}
then $q^i(x_0)=Q(x_0)$ and
\[
\uu(x_0+x)=\left(-x\cdot\nu(x_0)\right)_+\Aa_{x_0}+o(|x|),\quad\text{ as }x\ra0,
\]
for some vector $\Aa_{x_0}=(\alpha^1,\dots,\alpha^m)$ that 
\begin{equation}\label{thm:identification-q-norm-alpha}
|\Aa_{x_0}|_p^p=(\alpha^1)^p+\dots+(\alpha^m)^p=\frac1{p-1}Q(x_0)^p.
\end{equation}
\end{theorem}

\begin{proof}
Without loss of generality assume that $\nu(x_0)=\ee^n$. 
Let $\uu_k$ be a blow-up sequence with respect to balls $B_{r_k}(x_0)$, with blow-up limit $\uu_0$. 
Since $\nu(x_0)$ is the normal vector to $\partial\{|\uu|>0\}$ at $x_0$,
\[
\int_{B_r(x_0)}\left|\chi_{\{|\uu|>0|\}}-\chi_{\{x:(x-x_0)\cdot\nu(x_0)<0\}}\right|\,dx=o(r^n),\quad\text{ as }r\ra0.
\]
This along with \eqref{blow-up-limit-rel4} implies $\chi_{\{|\uu_0|>0\}}=\chi_{\{x_n<0\}}$ almost every where in $\bR^n$. 
By Lemma \ref{lem:blow-up-minimizer} we know that $\uu_0$ is an absolute minimizer of $J_0$ and so continuous.
Then $\{|\uu_0|>0\}=\{x_n<0\}$. This proves that $\{x_n=0\}$ is the topological tangent plane to $\partial\{|\uu|>0\}$ at $x_0$.
Now let
\[
\phi(x)=\min\left(1, \max(0,2-|x_n|)\right)\eta(x'),
\]
where $0\leq\eta\in C_0^\infty(B'_R)$ and $B'_R$ is $(n-1)$-dimensional ball with radius $R$ ($R$ is arbitrary and fixed). Denote $\phi_k(x):=r_k\phi(\frac{x-x_0}{r_k})$ and write
\begin{align*}
&-\int_{\bR^n}|\nabla u_0^i|^{p-2}\nabla u_0^i\cdot\nabla\phi\,dx = 
-\lim_{k\ra\infty}\int_{\bR^n}|\nabla u_k^i|^{p-2}\nabla u_k^i\cdot\nabla\phi \,dx 
\\ &= -\lim_{k\ra\infty}r_k^{-n}\int_{\bR^n}|\nabla u^i|^{p-2}\nabla u^i\cdot\nabla\phi_k \,dx 
= \lim_{k\ra\infty}r_k^{-n}\int_{\bR^n}\Delta_p u^i\phi_k \,dx \\
& = \lim_{k\ra\infty}r_k^{-n}\int_{\bR^n}q^i\phi_k\,d\cH^{n-1}\lfloor\,\partial\{|\uu|>0\}
= \lim_{k\ra\infty}\int_{\bR^n}q^i(x_0+r_kx)\phi(x)\chi_{\partial\{|\uu_k|>0\}}\,d\cH^{n-1} 
\\ &=\lim_{k\ra\infty}\int_{\bR^n}q^i(x_0)\phi(x)\chi_{\partial\{|\uu_k|>0\}}\,d\cH^{n-1}
= q^i(x_0)\int_{\{x_n=0\}}\eta(x')\,dx',
\end{align*}
where we have used  assumption \eqref{thm:identification-q-assump-lebeg-point} and property \eqref{blow-up-limit-rel3}.
Therefore, for any test function $\zeta\in C_0^\infty(B_R)$ we have
\[
-\int_{B_R\cap\{x_n<0\}}|\nabla u_0^i|^{p-2}\nabla u_0^i\cdot\nabla\zeta\,dx = q^i(x_0)\int_{B_R'}\zeta(x',0)\,dx'.
\]
Since $\Delta_p u^i_0=0$ in $\{x_n<0\}$, from boundary regularity it follows that 
\[
|\nabla u_0^i|^{p-2}\partial_n u_0^i=-q^i(x_0)\quad\text{ on }\{x_n=0\},
\]
in the classical sense. 
We need to show that 
\begin{equation}\label{thm:identification-q-rel1}
u_0^i(x)=\alpha^i(-x_n)_+, \quad \text{ where }\alpha^i:=\left(q^i(x_0)\right)^{1/(p-1)}.
\end{equation}
To see this, define  $w_0$  by
\[
w_0(x):=\left\{\begin{array}{ll}
u_0^i(x),&\, \text{ in }x_n\le 0,\\[8pt]
-u_0^i(x',-x_n),&\,   \text{ in }x_n>0.
\end{array}\right.
\]
It is obvious that $w_0$ is $p$-harmonic in whole $\bR^n$ as well as 
$$
\norm{\nabla w_0}_{L^\infty(\bR^n)} = \norm{\nabla u_0^i}_{L^\infty(\bR^n)}\leq \norm{\nabla u^i}_{L^\infty(B_r(x_0))},\quad\text{ for any }r>0.
$$
By Liouville's  theorem  
we conclude  that $w_0$ is a linear function. 
The boundary value on $x_n=0$, ($u_0^i=0$ and $\partial_n u_0^i=-\alpha^i$) shows that $w_0(x)=-\alpha^i x_n$.
This proves \eqref{thm:identification-q-rel1} and shows that
\[
u^i(x_0+x)=\alpha^i(-x_n)_++o(|x|),\quad\text{ as }x\ra0.
\]
We just have to show \eqref{thm:identification-q-norm-alpha}. To do this, note that $\uu_0$ is an absolute minimizer of $J_0$. 
Apply Lemma \ref{lem:domain-variation-formula} for $\uu_0$ and some $\Psi\in C_c^\infty(\bR^n)$
\[\begin{split}
0=&\int_{\{x_n<0\}}\rdiv\left[\sum_{i=1}^m \left(|\nabla u_0^i|^{p}\Psi-p|\nabla u_0^i|^{p-2}(\nabla u_0^i\cdot\Psi)\nabla u_0^i\right)+Q(x_0)^p\Psi\right]dx\\
=&\int_{\{x_n=0\}}\sum_{i=1}^m \left(|\nabla u_0^i|^{p}(\Psi\cdot\ee^n)-p|\nabla u_0^i|^{p-2}(\nabla u_0^i\cdot\Psi)\partial_n u_0^i\right)+Q(x_0)^p(\Psi\cdot\ee^n)\,d\cH^{n-1}\\
=&\int_{\{x_n=0\}}\left(\sum_{i=1}^m(1-p)(\alpha^i)^p+Q(x_0)^p\right)(\Psi\cdot\ee^n)\,d\cH^{n-1}.
\end{split}\]
Thus \eqref{thm:identification-q-norm-alpha} will be obtained.
\end{proof}


\medskip

\section{Regularity of free boundary}\label{sec:FB}

\begin{definition}
Let $\uu\in C(\Omega,\bR^m)$. We say that the boundary condition $\nabla|\uu|_p=g$ on $\partial\{|\uu|>0\}$ holds in viscosity sense, if 
\begin{enumerate}[$\circ$]
\item 
For every differentiable function $\phi:\bR^n\ra\bR$ that touches $|\uu|_p$ from below in some $x_0\in\partial\{|\uu|>0\}$, that is
\[
|\uu(x_0)|_p=\phi(x_0),\quad\text{ and } \quad |\uu|_p\ge \phi \ \text{ in } \{|\uu|>0\}\cap B_r(x_0)
\]
 for some $r>0$,  we have $|\nabla\phi(x_0)| \leq g(x_0)$.
 
 \item For every differentiable function $\phi:\bR^n\ra\bR$ that touches $|\uu|_p$ from above in some $x_0\in\partial\{|\uu|>0\}$, that is
\[
|\uu(x_0)|_p=\phi(x_0),\quad\text{ and } \quad |\uu|_p\le \phi \ \text{ in } \{|\uu|>0\}\cap B_r(x_0)
\]
 for some $r>0$,  we have $|\nabla\phi(x_0)| \geq g(x_0)$.
\end{enumerate} 
\end{definition}

\begin{lemma}\label{lem:bndry-viscosity}
Let $\uu$ be a (local) minimizer, 
then  the boundary condition
\[
\nabla|\uu|_p=\frac1{(p-1)^{1/p}}Q, \quad \text{ on }\partial\{|\uu|>0\}.
\]
holds in the viscosity sense.
\end{lemma}
\begin{proof}
We show that the boundary condition holds on every point $x_0\in\partial\{|\uu|>0\}$.  
Suppose $\phi$ touches $|\uu|_p$ from below at $x_0$. 
Consider the blow-up sequences
\[
\uu_k(x)=\frac{\uu(x_0+r_kx)}{r_k}\quad  \text{ and }\quad  \phi_k(x)=\frac{\phi(x_0+r_kx)}{r_k},
\] 
where  $r_k\searrow  0$. Observe that
\[
\phi_k(0)=|\uu_k(0)|_p=0,\quad \phi_k(x)\le |\uu_k(x)|_p, 
\]
and up to a subsequence we have 
\begin{equation}\label{lem:bndry-viscosity-rel1}
\nabla\phi(x_0)\cdot x=\lim_{k\ra\infty}\phi_k(x)\le \lim_{k\ra\infty}|\uu_k(x)|_p=|\uu_0(x)|_p\qquad\text{ in }\bR^n. 
\end{equation}
If $\nabla \phi(x_0)=0$, the  viscosity  condition holds trivially. 
Otherwise, the non-coincidence set $ \{|\uu_0|>0\} $ contains half-space $\{x:\nabla\phi(x_0)\cdot x>0\}$. 
On the other hand, $\uu_0$ is minimizer of $J_0$ (Lemma \ref{lem:blow-up-minimizer}) and by Lemma \ref{rem:coopertive-property}
every nontrivial component of $\uu_0$, say $u^i_0$,  is positive in $\{x:\nabla\phi(x_0)\cdot x>0\}$. 
According to  Lemma \ref{blow-up:linear-estimate-p-harmonic}, $u^i_0(x)=\alpha^i (\nabla \phi(x_0)\cdot x)+o(|x|)$ for some $\alpha^i$.
Thus any blowup of $\uu_0$ at $x=0$ must be of the form $\uu_{00}(x)=\Aa_0 (\nabla \phi(x_0)\cdot x)$ where $\Aa_0=(\alpha^1,\cdots,\alpha^m)$. 
Again apply Lemma \ref{lem:blow-up-minimizer} along with Lemma \ref{lem:domain-variation-formula}, we get
\[|\Aa_{0}|_p|\nabla\phi(x_0)|=\frac1{(p-1)^{1/p}}Q(x_0).\]
Thus \eqref{lem:bndry-viscosity-rel1} yields that 
\begin{align*}
\nabla\phi(x_0)\cdot x \le |\uu_0(x)|_p = |\Aa_0|_p|\nabla\phi(x_0)\cdot x| + o(|x|),
\end{align*}
and so
$$|\nabla\phi(x_0)|\leq   \frac1{(p-1)^{1/p}}Q(x_0).$$  
The same argument holds when $\phi$ touches $|\uu_0|_p$ from above.
\end{proof}

\begin{definition}\label{def:NTA}
A domain $\Omega\subset\bR^n$ is called non-tangentially accessible domain (NTA) with parameters $M\geq 1$ and $R_0>0$ if
\begin{enumerate}[(i)]
\item
$\Omega$ satisfies the corkscrew condition, that is, for any $x\in\partial\Omega$ and $r\in(0,R_0)$ there exists $a_r(x)\in \Omega\cap B_{r}(x)$ such that $M^{-1}r<\dist(a_r(x),\partial\Omega)$.

\item
$\bR^n\setminus\Omega$ satisfies the corkscrew condition. 

\item
If $x\in\partial\Omega$ and $x_1, x_2\in B_r(x)\cap\Omega$ for $0<r<R_0$, then there exists a rectifiable curve $\gamma:[0,1]\ra\Omega$ with $\gamma(0)=x_1$ and $\gamma(1)=x_2$ such that $\cH^1(\gamma)\le M|x_1-x_2|$ and
\[
\min\left\{\cH^1(\gamma([0,t])), \cH^1(\gamma([t,1]))
\right\}\le M\dist(\gamma(t),\partial\Omega),\quad\text{ for every }t\in[0,1],
\]
where $\cH^1$ denotes length or the one-dimensional Hausdorff measure. 
\end{enumerate}
\end{definition}

\begin{definition}
We say that $x_0\in \partial\{|\uu|>0\}$ is a {\it regular  point} of the  free boundary if for every $\epsilon>0$ there are $r<1$ and a vector $\Aa \in \bR^m$ and unit vector $\nu\in \bR^n$ such that
\[
\norm{\uu_{r,x_0}-(x\cdot \nu)_+\Aa}_{L^\infty(B_1)} \le \epsilon.
\]
 We denote the set of all regular points by $\cR_{\uu} $.
Theorem \ref{thm:identification-q} proves that $\cH^{n-1}(\partial\{|\uu|>0\}\setminus \cR_{\uu})=0$.
\end{definition}

\begin{theorem}\label{thm:reg-free-boundary}
Let $\uu=(u^1,\dots,u^m)$ be a (local) minimizer and  $x_0\in\cR_{\uu} $.
Furthermore, assume that $B_{r_0}(x_0)\cap\{|\uu|>0\}$ is NTA domain.
Then $\cR_{\uu} \cap B_r(x_0) $,  for some $0<r\leq r_0$, is $C^{1,\alpha}$ for a universal exponent $0<\alpha<1$.
\end{theorem}

\begin{proof}
We may assume that $u^1>0$ in $B_{r_0}(x_0)\cap\{|\uu|>0\}$. 
First we show that there is a H\"older function $g:B_r(x_0)\cap\partial\{|\uu|>0\}\ra[c,1]$ for some $0<r\leq r_0$ and $0<c\leq 1$ such that $u^1$ is a viscosity solution to the problem
\begin{equation}\label{scalar-viscosity-problem}
\begin{split}
\Delta_p u^1=0, &\quad\text{ in }\{u^1>0\}\cap B_r, \\
 |\nabla u^1|=\frac{gQ}{(p-1)^{1/p}}, &\quad\text{ on }\partial\{u^1>0\}\cap B_r.
\end{split}
 \end{equation}

Since $B_{r_0}(x_0)\cap\{|\uu|>0\}$ is NTA domain, 
the boundary Harnack  inequality (see \cite{lewis2008boundary}) implies that $g_i:=u^i/u^1$ is H\"older continuous in $\overline{\{|\uu|>0\}}\cap B_r$ for some $r\le r_0$.
Define
\[
g:=(1+g_2^p+\dots +g_m^p)^{-1/p},
\]
and observe that $u^1= g|\uu|_p$.
Suppose now the test function $\phi$ is touching $u^1$ from below in a point $y\in \partial\{|\uu|>0\}$. 
For $\rho$ small enough, choose a constant $C>0$ such that
\[
\frac1{g(x)}\ge \frac1{g(y)}-C|x-y|^\mu\ge0, \ \ \text{ for every }x\in \overline{\{|\uu|>0\}}\cap B_\rho,
\]
where $\mu$ is H\"older exponent of $g$.
Set $\psi(x)=\phi(x)(1/g(y)-C|x-y|^\mu)$, we get $\psi(y)=|\uu(y)|_p=0$ (note that since $\phi(y)=0$, so $\psi$ is differentiable at $y$) and 
\[
\psi(x)\le u^1(x)(\frac1{g(y)}-C|x-y|^\mu)=g(x)(\frac1{g(y)}-C|x-y|^\mu)|\uu(x)|_p\le |\uu(x)|_p.
\]
Therefore, $\psi$ touches $|\uu|_p$ at $y$ from below. By Lemma \ref{lem:bndry-viscosity}
\[
|\nabla \psi(y)|\le \frac{Q(y)}{(p-1)^{1/p}}.
\]
Note that $\nabla\psi(y)=\frac1{g(y)}\nabla \phi(y)$ to see the boundary condition \eqref{scalar-viscosity-problem} in viscosity sense.

The regularity of free boundary follows by the known results on the regularity of the one-phase scalar problem \eqref{scalar-viscosity-problem}; see \cite{ferrari2021regularity}.
\end{proof}

Our next result improves the theorem above, in the sense that we can remove the NTA conditions, for $p$ close to 2. The main reason for not being able to handle the NTA, is the lack of ACF-monotonicity formula, use in classical paper; see  proof of Proposition \ref{connected-coincidence-set}, to prove the connectivity argument for the NTA.

\begin{theorem}\label{thm:reg-free-boundary-2}
Let $\uu=(u^1,\dots,u^m)$ be a (local) minimizer. Then there is $\epsilon_0 > 0$, such that for any $p \in  (2-\epsilon_0, 2 + \epsilon_0)$ we have 
\begin{itemize}
\item[i)]    The regular set  $\cR_{\uu}  $, is locally   $C^{1,\alpha}$.
\item[ii)] 
In dimensions $2,3,4$ the free boundary is $C^{1,\alpha}$.

\end{itemize}
Here  $0 <\alpha  <1$ is  is  a universal exponent.
\end{theorem}

\begin{proof}
To prove  (i) it suffices, in virtue of Theorem \ref{thm:reg-free-boundary},
to show that the free boundary is NTA, when $\epsilon_0$ is small enough. This, however, is a consequence of Proposition 
\ref{lem:NTA-domain} in Appendix A, by choosing $\epsilon_0$ accordingly.

Turning to sstatement (ii) we recall  that  in these dimensions, and for $p=2$,  free boundaries are locally $C^{1,\alpha}$; for $n=2$ this was shown in \cite{alt1981existence}, for $n=3$ see \cite{CJK}, for $n=4$ see \cite{JS}.
Now for $p \approx 2$, the free boundary has to be close to that of the case $p=2$, and hence flat. More specifically, 
given a free boundary point $x_0$ for the $p$-problem, and $p$ close enough to 2, we have that in $B_r (x_0)$ the free boundaries of both $p$ and 2, have Huasdorff distance $\delta  \ll r $, and in particular the $p$-free boundary is $(\delta/2 )$-flat.
%
%
%
%
%
%
\end{proof}


\appendix
\section{Non-tangentially accessible domain (NTA)}\label{sec:NTA}

We note that in Definition \ref{def:NTA}, the condition $(iii)$ can be replaced by {\it Harnack chain condition} (see \cite{bennewitz2005dimension}),  i.e.,  
\begin{enumerate}
\item[(iii)'] Given $\e>0$, $x_1, x_2\in \Omega$ such that $\dist(x_i,\partial\Omega)\ge \e$, $i=1,2$ and $|x_1-x_2|\ge \tilde C\e$, we can find points $x_1=y_1,y_2,\cdots, y_\ell=x_2$ for which
\begin{enumerate}
\item $B_\e(y_i) \subset \Omega$ for $i=1,\cdots,\ell$.

\item $B_\e(y_i)\cap B_\e(y_{i+1})\ne\emptyset$ for $i=1,\cdots,\ell-1$.

\item The length of chain, $\ell$, depends on $\tilde C$ but not on $\e$.
\end{enumerate}
\end{enumerate}

We need an analogue of Theorem 4.1 in \cite{aguilera1987optimization} to show that the non-coincidence set is NTA.

\begin{lemma}\label{lem:growing-value}
Let $\uu$ be a minimizer and suppose $0<|\uu(x_0)|$. Define 
$\delta=\dist(x_0,\Gamma)$,  $\delta_1=\dist(x_0, \{|\uu|\le \frac12 |\uu(x_0)|_\infty\})$, and suppose $B(x_0,\delta)\subset\Omega$.
Then, there exist universal constants 
$\lambda>1 > \sigma $ such that 
\begin{enumerate}[(i)]
\item
$\sigma\delta\le \delta_1\le \delta$.
\item
For some $y\in \partial B_{\delta_1}(x_0)$, $|\uu(y)|_\infty\ge \lambda |\uu(x_0)|_\infty$.
\end{enumerate}
\end{lemma}

\begin{proof}
According to relation \eqref{corol:nondegeneracy},
\[
c_0\delta\le |\uu(x_0)| \le C_0\delta,
\]
and if $u^i(x_0)=\max_{1\le j\le m} u^j(x_0)=|\uu(x_0)|_\infty$,
\[
\frac{c_0\delta}{\sqrt{m}}\le u^i(x_0) \le C_0\delta.
\]
If $z\in \partial \{|\uu|\le \frac12 |\uu(x_0)|_\infty\}$ such that $|z-x_0|=\delta_1$, then
\[
\frac12 u^i(x_0)=\frac12 |\uu(x_0)|_\infty   \le |\uu(x_0)|-|\uu(z)| \le |\uu(x_0)-\uu(z)| \le C_0\delta_1,
\]
where we have used $|\uu(z)|= \frac12|\uu(x_0)|_\infty\le \frac12|\uu(x_0)|$.
Thus $(i)$ holds for $\sigma = \frac{c_0}{2C_0\sqrt m} $.

To see $(ii)$, note that $v(x):=u^i(x_0+\delta_1x)/u^i(x_0)$ is a $p$-harmonic function in  $B_{\delta/\delta_1}$ such that $v(0)=1$, $v(\hat z) \le \frac12$ for some $\hat z\in \partial B_1$. In addition, the Lipschitz constant of $v$ is bounded by $C_0\delta_1/u^i(x_0)\le  C_0\delta/u^i(x_0) \le C_0\sqrt{m}/c_0 = 1/2\sigma$.
We claim that there is a universal constant $\lambda>1$ such that $v(\hat y)\ge \lambda $ for some $\hat y\in \partial B_1$.
Otherwise, we find a sequence of $p$-harmonic functions $v_k$ with
\[
 \norm{v_k}_{L^\infty(B_1)}\le 1+\frac 1k, \quad \norm{\nabla v_k}_{L^\infty(B_1)}\le \frac 1{2\sigma}, \quad 
 v_k(0)=1,\quad v_k(z_k)\le \frac12,
\]
 for some $|z_k|=1$.
Then there is a subsequence converging in $C^{1,\alpha}(\overline B_1)$ to a $p$-harmonic $v_0$ where
$v_0(0)=1$, $v_0(z_0)\le 1/2$, $\norm{v_0}_{L^\infty(B_1)}=1$. This is a contradiction with the maximum principle.
Therefore, we have found $y\in \partial B_{\delta_1}(x_0)$ such that
\[
|\uu(y)|_\infty \ge u^i(y) \ge \lambda u^i(x_0) =\lambda |\uu(x_0)|_\infty.\qedhere
\]
\end{proof}

\begin{lemma}\label{lem:nondegeneracy-for-components}
There exists a universal constant $\tilde c_0$ such that if $x_0\in\partial\{|\uu|>0\}$,  $x_1\in B_{r/2}(x_0)$ and $A_r$ is the connected component of $\{|\uu|>\frac12|\uu(x_1)|_\infty\}\cap B_r(x_0)$ containing $x_1$, then 
\[
\norm{\uu}_{L^\infty(A_r)}\ge \tilde c_0r.
\]
\end{lemma}
\begin{proof}
We use Lemma \ref{lem:growing-value} to inductively define a sequence of points $x_1,x_2,\dots, x_k, x_{k+1}$ so that for $j=1,\dots, k$,
\begin{enumerate}[$(i)$]
\item
$|x_{j+1}-x_j|=\delta_j= \dist(x_j, \{|\uu|\le \frac12 |\uu(x_j)|_\infty\}),$
\item
$|\uu(x_{j+1})|_\infty \ge \lambda |\uu(x_j)|_\infty,$
\item
$B_{\delta_j}(x_j)\subset \{|\uu|>\frac12|\uu(x_1)|_\infty\}.$
\end{enumerate}
By $(ii)$, we know that this process cannot continue indefinitely without stepping out of $B_r(x_0)$. 
So, we stop at the first $k$ for which $B_{\delta_{k+1}}(x_{k+1})\not\subset B_r(x_0)$. 
Also, by Lemma \ref{lem:growing-value}, we know that 
\[
\delta_j \le \dist(x_j,\Gamma)\le \frac{\delta_j}\sigma,
\]
and therefore by \eqref{corol:nondegeneracy}, 
\[
c_0\delta_j \le |\uu(x_j)| \le \frac{C_0\delta_j}{\sigma}.
\]
Now applying  $(ii)$, we obtain  (recall that $\sigma = \frac{c_0}{2C_0\sqrt m} $)
\[
\delta_j \le \frac{\sqrt{m}}{c_0}|\uu(x_j)|_\infty \le  \frac{\sqrt{m}\lambda^{-\ell}}{c_0}|\uu(x_{j+\ell})|_\infty 
\le \frac{\lambda^{-\ell}}{2\sigma^2}\delta_{j+\ell}.
\]
Therefore, 
\[
|x_k-x_0|\le |x_1-x_0|+\sum_{j=1}^{k-1}|x_{j+1}-x_j| \le \frac r2+\sum_{j=1}^{k-1}\delta_j \le
\frac r2+\frac{\delta_k}{2\sigma^2} \sum_{j=1}^{k-1}\lambda^{j-k} \le \frac r2+\frac{\delta_k}{2\sigma^2(\lambda-1)}.
\]
On the other hand, 
\[\begin{split}
\delta_{k+1}=&\dist(x_{k+1}, \{|\uu|\le \frac12 |\uu(x_{k+1})|_\infty\}) \\
\le &\delta_k+\dist(x_{k}, \{|\uu|\le \frac12 |\uu(x_{k+1})|_\infty\}) \\
\le & \delta_k+\dist(x_{k}, \{|\uu|\le \frac12 |\uu(x_{k})|_\infty\}) = 2\delta_k.
\end{split}\]
Since $B_{\delta_{k+1}}(x_{k+1})\not\subset B_r(x_0)$, we get
\[
r\le |x_{k+1}-x_0|+\delta_{k+1} \le |x_k-x_0| + 3\delta_k \le \frac r2 + c\delta_k.
\]
Thus $\gamma r\le \delta_k$ for the universal constant $\gamma$. It necessitates that $B_{\gamma r}(x_k)\subset A_r$ and also
\[
|\uu(x_k)|\ge c_0\dist(x_k,\Gamma)\ge c_0\gamma r.
\]
In particular, 
\[
\max_{A_r}|\uu| \ge c_0\gamma r.\qedhere
\]
\end{proof}

\begin{proposition}\label{connected-coincidence-set}
There is $\e_0>0$ such that for any $p\in (2-\e_0,2+\e_0)$ there is no  global minimizer of $J_0$ (Lemma \ref{lem:blow-up-minimizer}) that $\uu(0)=0$ and $\{|\uu|>0\}\cap B_R$ is disconnected for every $R>0$.
\end{proposition}

\begin{proof}
First, for $p=2$. We apply the monotonicity formula and Lemma 4.4 in \cite{aguilera1987optimization}. 
Let $A_1$ and $A_2$ two different connected components of $\{|\uu_\infty|>0\}$. 
According to Lemma \ref{rem:coopertive-property}, we may choose a positive component of vector $\uu_\infty$ for each set $A_i$, $i=1,2$.
Thus we find a function $v$ which is harmonic in $A_1\cup A_2$ and vanishes in $\{|\uu_\infty|=0\}$. 
Also, since $\uu_\infty$ is a minimizer, we know that (Theorem \ref{thm:porosity})
\[
|B_r\setminus (A_1\cup A_2)| \ge c|B_r|.
\]
From here we infer that $\Phi(r)/r^\beta$ is a non-decreasing function of $r$ for some positive constant $\beta>0$ \cite[Lemma 4.4]{aguilera1987optimization}, where
\[
\Phi(r) = \frac1{r^4}\left(\int_{B_r\cap A_1}|\nabla v|^2 |x|^{2-n}\,dx\right)\left(\int_{B_r\cap A_2}|\nabla v|^2 |x|^{2-n}\,dx\right).
\]
Since $v$ is Lipschitz, say with constant $C_0$, we have the bound 
\[
\Phi(r) \leq C_0^4,
\]
and thus 
\[
\Phi(1) \le r^{-\beta}\Phi(r) \le C_0^4 r^{-\beta},
\]
which can not be valid for sufficiently large value of $r$. The contradiction proves the proposition when $p=2$.

To prove the proposition for $p$ close to 2 we argue by contradiction.  
Assume  that there is a sequence of global minimizers $\uu_i$ for $p_i\to 2$ that $\uu(0)=0$ and $\{|\uu_i|>0\}\cap B_R$ is not connected for any $R>0$.

We remark that all our results are stated in a slightly more general form, with constants depending uniformly for all $p\in [1/2, 3]$; see for the details \cite{danielli2006full}. 
In fact, we will have a uniform Lipschitz constant for all $p$ in this compact interval as well as the nondegeneracy constant. 
Also, the constant $c$ in Theorem \ref{thm:porosity} will be uniform. 
Therefore, we may choose a convergence subsequence $\uu_i\ra \tilde\uu_0$. 
By the same reasoning as the proof of Lemma \ref{lem:blow-up-minimizer} we get that $\tilde \uu_0$ is a minimizer for $p=2$.
On the other hand, $\{|\uu_0|>0\}\cap B_R$ is not connected for every $R$ which contradicts the later part of the proof. 
\end{proof}

\begin{lemma}\label{lem:uniform-connectivity}
Let $\e_0>0$ be the constant defined in Proposition \ref{connected-coincidence-set}. 
Then for any $p\in (2-\e_0,2+\e_0)$ there are constants $M\ge 1$ and $R_0>0$ such that if 
$x_0\in \partial\{|\uu|>0\}$, $y\in B_{R_0}(x_0)\cap \partial\{|\uu|>0\}$ and 
$x_1, x_2\in B_{r}(y)$ for some $r<R_0$,  
then
$\{|\uu|>d\}\cap B_{Mr}(y)$ has a connected component containing $x_1$ and $x_2$,
where $d=\frac12\min(|\uu(x_1)|_\infty, |\uu(x_2)|_\infty)$.
\end{lemma}

\begin{proof}
Fix constant $p$ and assume the contrary there are sequences $\partial\{|\uu|>0\}\ni y_i\to x_0$, $x_1^i,x_2^i \in B_{r_i}(y_i)$, $r_i\to0$ and $M_i\to\infty$ such that $x_1^i$ and $x_2^i$ are not connected in $\{|\uu|>d_i\}\cap B_{M_ir_i}(y_i)$ where  $d_i=\frac12\min(|\uu(x_1^i)|_\infty, |\uu(x_2^i)|_\infty)$.
Consider the blowup sequence 
\[
\uu(y_i+r_ix)/r_i\to \uu_0(x), \qquad d_i/r_i\to a,
\] 
Lemma \ref{lem:nondegeneracy-for-components} yields that $\{|\uu_0|>a\}\cap B_R$ has at least two connected components for any $R>0$. 
Note that $a<\infty$ due to the Lipschitz regularity. Now consider a blowdown of $\uu_0$
\[
\uu_0(R_i x)/R_i\to \uu_\infty(x), \qquad R_i\to\infty,
\]
then $\{|\uu_\infty|>0\}\cap B_R$ must have at least two connected components for any  $R>0$.
On the other hand, $\uu_\infty$ is a minimizer of $J_0$; Lemma \ref{lem:blow-up-minimizer}.
This contradicts Proposition \ref{connected-coincidence-set}.
\end{proof}

\begin{proposition}\label{lem:NTA-domain}
Let $\uu$ be a minimizer when  $p\in (2-\e_0,2+\e_0)$ and $\e_0$ is the constant defined in Proposition \ref{connected-coincidence-set}.
Suppose $x_0\in\partial\{|\uu|>0\}$. Then $\{|\uu|>0\}$ is NTA in a neighborhood of $x_0$.
\end{proposition}

\begin{proof}
{\bf Step 1:} (Property  $(i)$, the  corkscrew condition for $\{|\uu|>0\}$.) 
\\
Assume that $M>\norm{\nabla\uu}_{L^\infty(B_1)}/c$ where $c$ is the nondegeneracy constant defined in Lemma \ref{lem:nondegeneracy}. 
If the condition fails at a  point $x\in \partial\{|\uu|>0\}$, then for any $y\in B_r(x)\cap\{|\uu|>0\}$ we must have  $\dist(y,\partial\{|\uu|>0\})\le M^{-1}r$. Thus 
$$|\uu(y)|\le M^{-1}r\norm{\nabla\uu}_{L^\infty(B_1)}\le cr$$
 and by nondegeneracy,  Lemma \ref{lem:nondegeneracy}, $\uu=0$ in $B_{\kp r}(x)$. It contradicts that $x\in \partial\{|\uu|>0\}$.

\medskip

\noindent
{\bf Step 2:} (Property  $(ii)$, the corkscrew condition for $\{|\uu|=0\}$.)\\
Assume the contrary, $\{|\uu|=0\}$ does not satisfy the corkscrew condition in any neighborhood of $x_0$ for any constant $M$. 
Thus there is a sequence $x_j\ra x_0$ and $r_j\ra0$ such that we can not find point $a_{r_j}(x_j)$ with the desired property. 
On the other hand, Theorem \ref{thm:porosity} infer that the interior of $\{|\uu|=0\}$ is nonempty. 
Let $B_{\tau_j}(y_j)$ be the biggest ball inside  $B_{r_j}(x_j)\cap\{|\uu|=0\}$, then we must have $\tau_j/r_j\ra0$.
Now consider the blowup $\uu_j(x):=\uu(x_j+r_j x)/r_j\ra \uu_0(x)$, it will be a minimizer whose coincidence set has no interior in $B_1$. 
This contradicts Theorem \ref{thm:porosity}.

\medskip

\noindent
{\bf Step 3:} (Harnack chain condition.)\\
Suppose that $x_1$ and $x_2$ are such that for some $\tilde C>0$ and $\e>0$ we have 
\[
|x_1-x_2| <\tilde C\e, \qquad B_\e(x_i)\subset\{|\uu|>0\}, \ i=1,2.
\]
We may assume without loss of generality, $\dist(x_1,\partial\{|\uu|>0\})\le \dist(x_2,\partial\{|\uu|>0\})=\delta_0$. 
If $\delta_0\ge \tilde C\e$, then $x_1\in B_{\tilde C\e}(x_2)\subset \{|\uu|>0\}$ and we can easily find the Harnack chain. 
So, consider the case $\delta_0< \tilde C\e$ and choose $x\in \partial\{|\uu|>0\}$ such that $|x-x_2| =\delta_0$. 
Then $x_1,x_2\in B_r(x)$ for $r= 2\tilde C\e$.

By Lemma \ref{lem:uniform-connectivity}, $\{|\uu|>d\}\cap B_{Mr}(x)$ has a connected component containing $x_1$ and $x_2$, where $d=\frac12\min(|\uu(x_1)|_\infty, |\uu(x_2)|_\infty)$.
Now we have a curve $\gamma:[0,1]\ra \{|\uu|>d\}\cap B_{Mr}(x)$ having $x_1$ and $x_2$ as end point. 
For every $t\in[0,1]$ we know that
\[
|\uu(\gamma(t))| \ge d \ge \frac{c_0\e}{2\sqrt{m}},
\]
where $c_0$ comes from \eqref{corol:nondegeneracy}. Hence,
\[
\dist(\uu(\gamma(t)),  \partial\{|\uu|>0\})\ge \sigma\e,
\]
where $\sigma=\frac{c_0}{2C_0\sqrt{m}}$.
Now we can find a sequence $y_1,\cdots, y_\ell$ on the image of $\gamma$ such that 
\[
\gamma[0,1]\subset\bigcup_{i=1}^\ell B_{\sigma\e}(y_i)\subset \{|\uu|>0\}\cap B_{Mr+\sigma\e}(x).
\]
Since $Mr+\sigma\e=(2M\tilde C+\sigma)\e$, the number of balls in covering, $\ell$ can be bounded by a constant depending only on the dimension and $(2M\tilde C+\sigma)/\sigma$, but not on $x_1, x_2$ or $\e$.
\end{proof}



\section{ An approximation lemma}\label{sec:App-approx}

Here we  prove a  lemma, which we  used in Lemma \ref{lem:bndry-viscosity}, and is   generalization of Lemma A.1 in \cite{caffarelli1989harnack} to any $1<p<\infty$ (see also Lemma A.1 in \cite{danielli2003singular}).

\begin{lemma}\label{blow-up:linear-estimate-p-harmonic}
Let $u$ be a nonnegative Lipschitz function in $B_1^+$ and assume that it is $p$-harmonic in $\{u>0\}$ and $u(0)=0$ ($1<p<\infty$). Then it has the asymptotic development 
\[
u(x)=\alpha x_n+o(|x|),\ \ \text{ as }x\ra0,
\] 
for some $\alpha\ge0$, if either
\begin{enumerate}[(i)]
\item
$u$ vanishes on $\{x_n = 0\}$, or

\item $\{x_n >0\} \subset \{u>0\}$.
\end{enumerate}
\end{lemma}
\begin{proof}
Part $(i)$ is Lemma A.1 in \cite{danielli2003singular}. The proof of $(ii)$ is also similar by a slight modification.
Let $\ell_k:=\sup\{l: lx_n\leq u(x)\ \text{ in }\ B_{2^{-k}}^+\}$. 
Since $\ell_k$ is a nondecreasing sequence and bounded by the Lipschitz constant of $u$.
Suppose $\alpha=\lim_{k\ra\infty}\ell_k$, then
\[
u(x)\ge \alpha x_n+o(|x|).
\]
If the claim fails there exists a sequence $x^k\ra0$ such that
\[
u(x^k)\ge \alpha x_n^k+\delta_0|x^k|,
\]
for some $\delta_0>0$. Define $u_k(x):=u(r_kx)/r_k$ where $r_k=|x^k|\ra0$. 
We may also assume that $r_k\le 2^{-k}$.
Since $u_k$ are uniformly Lipschitz, we may consider the blowup $u_0=\lim_{k\ra\infty}u_k$, as well as $x^k/r_k\ra x^0$, $|x^0|=1$. 
From the construction we will have $\alpha x_n\leq u_0(x)$ in $B_1^+$, and 
\[
\frac{\delta_0}2+\alpha x_n \le u_0(x)\quad\text{ and }\quad \frac{\delta_0}2+\ell_k x_n \le u_k(x)\quad\text{ in }B_\e(x^0),
\]
for a sufficiently small $\e>0$ and large $k$. 
Let now $w_k$ be a $p$-harmonic function in $B_1^+$ with smooth boundary values 
\[\begin{array}{cl}
w_k=\ell_k x_n&\ \text{ on }\partial B_1^+\setminus B_{\e/2}(x^0),\\[8pt]
w_k=\ell_k x_n + \frac{\delta_0}{4} &\ \text{ on }\partial B_1^+\cap B_{\e/4}(x^0),\\[8pt]
\ell_k x_n\le w_k \le \ell_k x_n + \frac{\delta_0}{4} &\ \text{ on }\partial B_1^+\cap B_{\e/2}(x^0),\\[8pt]
w_k=0&\ \text{ on }\{x_n=0\}\cap B_1.
\end{array}\]
From the comparison principle we will have $ w_k\le u_k$ in $B_1^+$. (Note that $u_k(x)\ge \ell_kx_n$, since $r_k\le 2^{-k}$).
Furthermore, $w_k\ra w_0$  in $C^{1,\sigma}(B_{1/2}^+)$ where $w_0$ is $p$-harmonic with boundary data $w_0=0$ on $\{x_n=0\}$ and $w_0\ge \alpha x_n$ on $\partial B_1^+$. 
By Hopf boundary principle,  
\[
w_0(x) \ge (\alpha+\mu)x_n\quad\text{ in }B_\gamma^+,
\]
for some small $\mu$ and $\gamma$. Thus for $x\in B_\gamma^+$,
\[\begin{split}
u_k(x)\ge w_k(x) &\ge w_0(x) - x_n \norm{\nabla(w_k-w_0)}_{L^\infty(B_{1/2}^+)}\\
 &\ge  \left(\alpha+\mu- \norm{\nabla(w_k-w_0)}_{L^\infty(B_{1/2}^+)}\right)x_n \\
 &\ge \left(\alpha+\mu/2\right)x_n 
\end{split}\]
Now returning to $u$ we get
$u(x)\ge (\alpha+\mu/2) x_n$ in $B_{\gamma/r_k}^+$. 
This is a contradiction with the definition of $\ell_k$ when $k$ is sufficiently large.
\end{proof}

\section*{Declarations}

\noindent {\bf  Data availability statement:} All data needed are contained in the manuscript.

\medskip
\noindent {\bf  Funding and/or Conflicts of interests/Competing interests:} The authors declare that there are no financial, competing or conflict of interests.

%
%
%

\end{document}